\documentclass[12pt]{amsart}
\usepackage{amscd,amssymb,amsthm,amsmath,amssymb,textcomp,supertabular,rotating,longtable,enumerate,rotating,mathrsfs,mathtools,multirow}
\usepackage[matrix,arrow,curve]{xy}
\usepackage{hyperref}

\usepackage{array}
\newcolumntype{H}{>{\setbox0=\hbox\bgroup}c<{\egroup}@{}}

\newcommand{\QQ}{{\mathbb Q}}
\newcommand{\ZZ}{{\mathbb Z}}

\newcommand{\PP}{{\mathbb P}}
\newcommand{\CC}{{\mathbb C}}

\newcommand{\LL}{{\mathbb L}}
\newcommand{\HH}{{\mathbb H}}
\newcommand{\Aff}{{\mathbb A}}

\newcommand{\dvA}{{\mathrm A}}
\newcommand{\dvD}{{\mathrm D}}

\newcommand{\cC}{{\mathcal{C}}}
\newcommand{\cL}{{\mathcal{L}}}
\newcommand{\cN}{{\mathcal{N}}}
\newcommand{\cO}{{\mathcal{O}}}

\DeclareMathOperator{\Pic}{Pic}
\DeclareMathOperator{\Spec}{Spec}
\DeclareMathOperator{\Proj}{Proj}
\DeclareMathOperator{\Ext}{Ext}
\newcommand{\Db}{{\mathcal{D}^b_{\mathrm{coh}}\,}}
\newcommand{\rk}{\mathrm{rk}\,}

\newtheorem{theorem}[equation]{Theorem}
\newtheorem*{theorem*}{Theorem}

\newtheorem{proposition}[equation]{Proposition}
\newtheorem*{proposition*}{Proposition}
\newtheorem{lemma}[equation]{Lemma}
\newtheorem{corollary}[equation]{Corollary}
\newtheorem*{corollary*}{Corollary}
\newtheorem{conjecture}[equation]{Conjecture}

\newtheorem*{problem*}{Problem}

\newtheorem*{question*}{Question}

\newtheorem*{construction*}{Construction}
\newtheorem*{maintheorem*}{Main Theorem}

\theoremstyle{definition}
\newtheorem{example}[equation]{Example}
\newtheorem*{example*}{Example}
\newtheorem{definition}[equation]{Definition}
\newtheorem*{definition*}{Definition}
\newtheorem{notation}[equation]{Notation}

\theoremstyle{remark}
\newtheorem{remark}[equation]{Remark}
\newtheorem*{remark*}{Remark}

\makeatletter\@addtoreset{equation}{section} \makeatother

\author{
Victor Przyjalkowski}

\title[Landau--Ginzburg models and exceptional collections]{Landau--Ginzburg models for Fano threefolds of Picard rank one and exceptional collections}

\address{\emph{Victor Przyjalkowski}
\newline
\textnormal{Steklov Mathematical Institute of Russian Academy of Sciences, 8 Gubkina street, Moscow 119991, Russia.}
\newline
\textnormal{\texttt{victorprz@mi-ras.ru, victorprz@gmail.com}}}

\begin{document}

\begin{abstract}
We study fibers with isolated singularities of Landau--Ginzburg models for Fano threefolds of Picard rank one.
We compare the data we get with maximal known lengths of exceptional collections in derived categories
of coherent sheaves on the Fano threefolds, verify some predictions of Homological Mirror Symmetry,
and present some expectations about exceptional collections for Fano threefolds.
\end{abstract}

\maketitle

\section{Introduction}
\label{section:introduction}

One of ways to study smooth complex projective variety $X$ is to study its bounded derived category of coherent sheaves $\Db(X)$.
For instance, considering $\Db(X)$ as a monoidal category, one can reconstruct $X$ from it~\cite{Ba02};
if $X$ is a Fano variety, 
then, by Bondal--Orlov reconstruction
theorem~\cite{BO01}, one can reconstruct $X$ even without considering a monoidal structure. On the other hand, to study derived
categories, it is often useful to ``simplify'' them, constructing exceptional collections of maximal length.

Interest to derived categories of coherent sheaves raised up after appearance of Kontsevich's Homological Mirror Symmetry.
Mirror Symmetry phenomenon relates symplectic and algebraic properties of objects of different nature;
in particular, it relates a smooth Fano variety $X$ to its Landau--Ginzburg model $(Y,w)$ --- a certain one-dimensional
family of Calabi--Yau varieties. Homological Mirror Symmetry claims that $\Db(X)$ is equivalent to the Fukaya--Seidel category
$\mathrm{FS} (Y,w)$, whose objects are vanishing (to singularities of singular fibers) Lagrangian cycles in fibers.
Ordinary double points in fibers $w^{-1}(t)$, $t\in \CC$, correspond to exceptional objects in $\Db(X)$.
The simplest case is when $(Y,w)$ is a Lefschetz pencil, that is when all fibers are smooth except for
$k$ fibers which have a single ordinary double point in each one.
In~\cite{Se08} a subca\-te\-go\-ry~$\mathcal{C}$ of $\mathrm{FS}(Y,w)$ was constructed, and it was proved that $\mathcal{C}$ has
a semiorthogonal decomposition into $k$ objects:
$$
\mathcal{C}=\langle \mathcal{L}_1,\ldots,  \mathcal{L}_k\rangle.
$$
According to~\cite{GPS20, GPS24a, GPS24b},  
the category $\mathcal{C}$ is ``large enough''.
This imp\-lies~$\mathcal{C}$ is equivalent to $\mathrm{FS}(Y,w)$.
Thus, in the case when $(Y,w)$ is a Lefschetz pencil, the Homological Mirror Symmetry conjecture predicts that $\Db(X)$ has a full exceptional collection consisting
of $k$ objects.

Similarly, $\mathrm{FS}(Y,w)$ has a full exceptional collection if the singularities of all singular fibers of $(Y,w)$
are ordinary double points (possibly lying in the same fiber).
Unfortunately, a Landau--Ginzburg model of a Fano variety rarely has only ordinary double points as singularities of fibers.
However if there are $k$ ordinary double points, then in~\cite{AAK16}
it was constructed a subca\-te\-go\-ry~$\mathcal{C}$ of $\mathrm{FS}(Y,w)$
such that
$$
\mathcal{C}=\langle \mathcal{R}_Y, \mathcal{L}_1,\ldots,  \mathcal{L}_k\rangle,
$$
where $\mathcal{R}_Y$ is some category, and $\mathcal{L}_i$ form an exceptional collection. It is expected
that $\mathcal{C}$ is ``large enough'' which implies that
$
\mathcal{C}$ is equivalent to $FS(Y,w)
$.
(Proof of this expectation has big symplectic difficulties though.)
In particular, the Homological Mirror Symmetry conjecture implies that $\Db(X)$ has an exceptional collection
with $k$ exceptional objects as well.

Given a Fano variety $X$ of dimension $n$, one can associate with it a so-called \emph{toric Landau--Ginzburg model} $f_X$,
that is, a certain Laurent polynomial in $n$ variables whose geometrical and combinatorial properties reflect ones of $X$,
see~\cite{Prz18} for details. A Landau--Ginzburg model $(Y,w)$ for $X$ is constructed as a Calabi--Yau
compactification of the family $f_X\colon (\CC^*)^n\to \CC$. If $n=3$ and $\Pic(X)=\ZZ$, and if we require that the Dolgachev--Nikulin
duality for anticanonical sections of $X$ and fibers of $w$ holds (see~\cite{DHKOP24}), then $(Y,w)$ is unique up
to flops (while there is a lot of different toric Landau--Ginzburg models for $X$).
By~\cite{Prz13}, the family $(Y,w)$ has at most one reducible or non-reduced fiber (the number
of its irreducible components is equal to $h^{1,2}(X)$). If there indeed exist a fiber with non-isolated singularities
(this happens  if and only if $h^{1,2}(X)>0$), we call it \emph{central}.
Thus, singular fibers of $(Y,w)$ are at most one fiber with non-isolated singularities,
and several fibers with isolated singularities. The number and type of these singularities depend only on $X$. 

\begin{notation}
\label{notation:epsilon}
Let $X$ be a smooth Fano threefold, and let $(Y,w)$ be its Landau--Ginzburg model.
We denote the number of isolated singularities of fibers of $w$ by $\varepsilon(X)$.
\end{notation}

Note that by Theorem~\ref{theorem:ODP for LG} all isolated singular points of fibers of $w$ are
ordinary double points (this also follows from the fact that Dolgachev--Nikulin duality implies that the fibers $w^{-1}(t)$ are K3 surfaces of Picard rank $19$).

Let us 
consider two- and three-dimensional cases. For this for a given smooth Fano variety $X$
let $i_X$ denote its Fano index, that is, the maximal integer dividing $-K_X$ in $\Pic(X)$.
Let us also denote the maximal possible length of exceptional collection in $\Db(X)$ by $\omega(X)$, and the maximal possible length of an extension of the \emph{standard} exceptional collection $\cO_X,\ldots,\cO_X(i_X-1)$ by $\omega_\cO(X)$, see Proposition~\ref{proposition:standard collection} and Definition~\ref{definition:omegi}.

Let $S_d$ be smooth two-dimensional Fano variety (i.\,e. del Pezzo surface) of degree $d$.
That is, $S_d$ is either a quadric surface (with $d=8$)
or a blow-up of $\PP^2$ in $9-d$ general enough points, where $d=0,\ldots,8$.
Recall that Landau--Ginzburg models for $S_d$ correspond to elements of $\Pic(S_d)\otimes \CC$.
We call Landau--Ginzburg models corresponding to general member by \emph{general} ones,
and a Landau--Ginzburg model correspond to $-K_{S_d}$ by \emph{anticanonical} one (for details see \cite{AKO06} and~\cite{Prz17}).

\begin{theorem}[{\cite{AKO06}, see also~\cite{Prz17}}]
\label{theorem: LG for del Pezzo}
Let $(Y,w)$ be a general Landau--Ginzburg model of $S_d$.
Then the singular fibers of $(Y,w)$ are $12-d$ fibers having a single ordinary double point in each.
\end{theorem}

Note that when we deform a general Landau--Ginzburg model to the anticanonical one most of isolated singularities of fibers of the models usually collide making a reducible fiber.

On the other hand, 
a smooth del Pezzo surface has a full exceptional collection
which is an extension of the standard one.
For a quadric surface this follows from by~\cite{Ka88}, 
and for blow ups of $\PP^2$ this is given by Orlov's blow-up theorem.
Moreover, the following holds.

\begin{theorem}[{\cite[Theorem 6.11]{KO95}}]
\label{theorem:del Pezzo}
Any exceptional collection in the derived category of del Pezzo surface is a part of full exceptional collection.
\end{theorem}

Theorems~\ref{theorem: LG for del Pezzo} and~\ref{theorem:del Pezzo}, and Proposition~\ref{proposition:standard collection} imply
the following.

\begin{corollary}
\label{corollary:anticanonical LG for del Pezzo}
Let $\varepsilon_{gen}(S_d)$ be the number of fibers of ordinary double points of general Landau--Ginzburg model for $S_d$.
One has $\omega(S_d)=\omega_\cO(S_d)=\varepsilon_{gen}(S_d)=12-d$. 
\end{corollary}

Now let us consider the threefold case. We refer to a smooth member of a family number $n$ from~\cite[Table in \S12.2]{IP99} of Picard rank one Fano threefolds as $X_{n}$.
(Note that the order of the families in the original classification paper differs from the usual one we use, see~\cite[Table 6.5]{Is78}.)
For readers' convenience we denote a smooth Fano threefold of index one and degree $k$ by $V_k$, of index $2$ and degree $8m$
by $Y_m$, and the quadric by $Q$.
In this paper we prove the following.

\begin{theorem}
\label{theorem:ODP for LG}
Let $X=X_n$ be a smooth Fano threefold of Picard rank one. The isolated singularities
of fibers of Calabi--Yau compactification $(Y,w)$ of a standard toric Landau--Ginzburg model for $X$
are ordinary double points, and their number $\varepsilon(X)$ is
\begin{itemize}
  \item $\varepsilon(X)=1$ if $n=1-4$, i.\,e. for  $V_2,V_4,V_6,V_8$;
  \item $\varepsilon(X)=2$ if $n=5-9,11-14$, i.\,e. for $V_{10},V_{12}, V_{14},V_{16}, V_{18}, Y_1,Y_2,Y_3,Y_4$;
  \item $\varepsilon(X)=4$ if $n=10,15,16,17$, i.\,e. for $V_{22},Y_5,Q,\PP^3$.
\end{itemize}
  Moreover, in the latter case $(Y,w)$ is a Lefschetz pencil. 
\end{theorem}

In Table~\ref{table:estimates} we present known estimates for maximal lengths of exceptional collections.
This, together with Lemma~\ref{lemma:omega for index 2}, the references in Table~\ref{table:estimates},
and Proposition~\ref{proposition:quadric and cubic}, gives the following.

\begin{corollary}
\label{corrollary: coincidence}
Let $X=X_n$ be a smooth Fano threefold of Picard rank one. One has
$$
4\ge \omega(X)\ge\omega_\cO(X)\ge \varepsilon (X).
$$
Moreover, $\omega(X)=\varepsilon (X)$
for $n=10,15,16,17$ (that is for $V_{22},Y_5,Q,\PP^3$, that is when $\Db(X)$ has a full exceptional collection), and $\omega_\cO(X)= \varepsilon (X)$ for $n=11,12,13,14$ (that is for $Y_1,Y_2,Y_3,Y_4$).
On the other hand, for $X=X_3$ (that is, for $V_6$) one has $\varepsilon(X)=2$, but it is expected that $\omega(X)=\omega_\cO(X)=2$ for general $X$,
while $\omega(X)=\omega_\cO(X)=1$ for certain special varieties in the deformation family of $X_6$.
\end{corollary}

Similar estimates hold for Fano complete intersections in any dimension. Recall
that in higher-dimensional case there is no known unique way to correspond Landau---Ginzburg
model to a Fano variety. However there is Givental's construction for Fano complete intersections
in toric varieties, which for the case of complete intersections in projective spaces
are birational to toric Landau--Ginzburg models, see, for instance,~\cite{ILP13} and~\cite{PSh15}.
We call them \emph{toric Landau--Ginzburg models of Givental's type}.

\begin{proposition}[{see~\cite{Ka88} and~\cite{Hu22}}]
\label{proposition:quadrics exceptional}
Let $X$ be a quadric hypersurface, and let $(Y,w)$ be a Calabi--Yau compactification
of its toric Landau--Ginzburg model of Givental's type.
The singularities of fibers of $(Y,w)$ are ordinary double points;
moreover, if $\varepsilon(X)$ is the number of such points, then
$$
\omega(X)=\omega_\cO(X)=\varepsilon(X).
$$
If $\dim(X)$ is odd, then all singular fibers of $w$ have a single singular point,
and if $\dim(X)$ is even, then all singular fibers have a single singular point
except for one fiber with two points.
\end{proposition}

\begin{remark}
In particular, since fibers $(Y,w)$ has ordinary double points as singularities, the categories $\Db(X)$ and $\mathrm{FS}(Y,w)$ have full exceptional collections of the same length.
\end{remark}

\begin{remark}
Let $X$ be a quadric hypersurface of dimension $n$. By~\cite{Ka88}, if $n$ is odd, then
$$
\Db(X)=\langle\cL, \cO_X,\ldots, \cO_X(n-1)\rangle,
$$
and if $n$ is even, then
$$
\Db(X)=\langle\cL_1, \cL_2, \cO_X,\ldots, \cO_X(n-1)\rangle,
$$
where $\cL_i$ are spinor bundles, and the exceptional pair $\langle \cL_1,\cL_2\rangle$ is completely orthogonal, that is $\Ext^k(\cL_i, \cL_j)=0$ for any $k$ and $i\neq j$.
\end{remark}

\begin{theorem}[\cite{Prz25}]
\label{theorem:complete intersections}
Let $X$ be a smooth Fano complete intersection of index $i_X$, 
and let $(Y,w)$ be a Calabi--Yau compactification
of its toric Landau--Ginzburg model of Givental's type.
The singular fibers of $(Y,w)$ are one fiber with possibly (and usually) non-isolated singularities,
and $i_X$ fibers with a single ordinary double point.
\end{theorem}

Note that by Definition~\ref{definition:omegi} one has $\omega_\cO(X)\ge i_X$.

Recall that Landau--Ginzburg models correspond to Fano varieties $X$ together with complexified symplectic forms (or, the same,
elements of $\Pic(X)\otimes \CC$) on $X$.
By~\cite{DHKOP24}, for smooth Fano threefold $X$ the deformation space of Landau--Ginzburg models is unobstructed and has dimension $\rk \Pic (X)-1$.
We call a general member of the space \emph{general Landau--Ginzburg model}, and the one that correspond to an anticanonical divisor
by \emph{anticanonical Landau--Ginzburg model}.
Since for Picard rank one case this space is just a point, in Theorem~\ref{theorem:ODP for LG} we considered a single
Landau--Ginzburg model for each family of Fano varieties.
The following example shows the importance of considering not anticanonical Landau--Ginzburg models,
but general ones (cf.~\cite{LP25}).

\begin{example}[{\cite[Example 3.1.5]{LP25}}]
\label{example:number of ODP decreases}
Consider a Fano threefold which is a blow-up of $\PP^3$ in genus $3$ and degree $6$ curve (which is an intersection
of $3$ cubics). One can see that, for a general Landau–Ginzburg model, there are five singular fibers: one simple normal crossings reducible fiber and four fibers each containing a single ordinary double point. However, as a general Landau--Ginzburg model varies toward the anticanonical Landau--Ginzburg model, one of the ordinary double points collides with the reducible fiber. Note that by Bondal--Orlov blow-up theorem one has $\omega(X)\ge 4$. We expect that $\omega(X)=\omega_\cO(X)=4$.
\end{example}

\begin{proposition}[{\cite[Proposition 6.13]{KP21} and~\cite[proof of Proposition 4.2]{KS25}}]
\label{proposition:quadric and cubic}
Let $X=X_3$, that is $X$ is an intersection of a quadric and a cubic.
If the quadric is smooth, then $\omega_\cO(X)\ge 2$; otherwise $\omega_\cO(X)\ge 1$.
\end{proposition}

Corollary~\ref{corollary:anticanonical LG for del Pezzo}, Corollary~\ref{corrollary: coincidence}, Example~\ref{example:number of ODP decreases}, and Proposition~\ref{proposition:quadric and cubic} justify the following.

\begin{conjecture}
\label{conjecture: collections for threefolds}
Let $X$ be a smooth Fano threefold. 
The maximal length of exceptional collection on $X$ equals the number of fibers
with ordinary double points for \emph{generic} Landau--Ginzburg model 
 except for the case $X=X_3$, that is for $V_6$, when the
equality holds for special $X$ in its deformation family. Moreover, $\omega(X)= \omega_\cO(X)$.
\end{conjecture}

Note that in higher dimensions or for the non-Fano case the situation may be more complicated.
That is, exceptional objects in the Fukaya--Seidel category may appear not only from ordinary
double points in fibers. Proposition~\ref{proposition:quadric and cubic} shows that this holds already for the Fano threefold case $X=X_3$.
Moreover, even if a variety has a full exceptional collection in the derived category,
its Landau--Ginzburg model may have not only ordinary double points as singularities of fibers.
Indeed, if fibers of a Landau--Ginzburg model have only ordinary double points,
then we expect that exceptional objects in the Fukaya--Seidel category of the Landau--Ginzburg model
correspond to the singular points, and any exceptional collection can be extended to a full one.
Thus, by Homological Mirror Symmetry conjecture, the same should hold for the derived category of the variety.
However the following example of threefold weak Fano variety shows that this may be not true.

\begin{example}[\cite{Kuz13}]
\label{example: different exceptional length}
Let $X$ be $\PP^3$ blown up in a line and the proper preimage of a conic transversally intersecting the line.
By Orlov's blow-up theorem $X$ has a full exceptional collection, and $\omega_\cO(X)=\omega(X)=8$.
However there exist an exceptional collection of four objects $\cL_i$ in $\Db(X)$ such that the category $\cC$ defined by
$$
\Db(X)=\langle\cC, \cL_1,\cL_2,\cL_3,\cL_4,\cL_5,\cL_6\rangle
$$
does not have an exceptional object.
\end{example}

\section*{Acknowledgements}
The author is very grateful to Alexander Kuznetsov for numerous useful discussions.
This work was supported by the Russian Science Foundation under grant no. 25-11-00057, \url{https://rscf.ru/project/25-11-00057/}.

\section{Exceptional collections}
\label{section:EC}

Let $X$ be a smooth complex Fano variety. Recall that $i_X$ denotes its Fano index.

\begin{proposition}[{see, for instance,~\cite[Lemma 3.4]{Kuz09}}]
\label{proposition:standard collection}
There is a semiorthogonal decomposition
$$
\Db(X)=\langle \cC,\cO_X,\ldots,\cO_X(i_X-1)\rangle.
$$
\end{proposition}

Given a Fano variety $X$, let us call the exceptional collection $\cO_X,\ldots,\cO_X(i_X-1)$ \emph{standard}.
In some cases one can extend the standard exceptional collection to a longer one; this happens exactly when the category
$\cC$ in Proposition~\ref{proposition:standard collection} has an exceptional object. In the paper we consider the
following two invariants of a variety.

\begin{definition}
\label{definition:omegi}
Let $X$ be a Fano variety. We define $\omega(X)$ as the maximal length of an exceptional collection in $\Db(X)$,
and $\omega_\cO(X)$ as the maximal length of 
an exceptional collection of type
$$
\cL_1,\ldots,\cL_k,\cO_X,\ldots,\cO_X(i_X-1).
$$
\end{definition}

\begin{remark}
\label{remark: omega ge omegaO}
Obviously  for any smooth Fano variety $X$ of dimension $n$ one has $\omega(X)\ge \omega_\cO(X)$. Moreover, by additivity of Hochschild homology under semiorthogonal decompositions (see~\cite[Corollary 7.5]{Kuz09b}) one has $\omega(X)\le \dim HH_0(X)$,
and by the Hochschild--Kostant--Rosenberg theorem one has $\dim HH_0(X)=\sum_{p=0}^{n} h^{p,p}(X)$. In particular, 
by Proposition~\ref{proposition:standard collection} one has
$$
i_X\le \omega_\cO(X) \le \omega(X)\le \sum_{p=0}^{n} h^{p,p}(X).
$$
\end{remark}

Theorem~\ref{theorem:del Pezzo} shows that 
for a del Pezzo surface $X$ one has
$$
\omega_\cO(X)=\omega(X)=\sum_{p=0}^{n} h^{p,p}(X).
$$
Higher-dimensional case is more challenging. Let $X$ be a smooth Fano variety of dimension $n$ such that $\Db(X)$ has a full exceptional collection. Then by additivity of Hochschild homology under semiorthogonal decompositions
one has
$\omega(X)=\sum_{p=0}^{n} h^{p,p}(X)$.
If $n=3$ and $\Pic(X)=\ZZ$, that is, if $X$ lies in the families $10$, $15$, $16$, $17$, i.\,e. $X$ is $V_{22}$, $Y_5$, $Q$, or $\PP^3$,
then $\sum_{p=0}^{3} h^{p,p}(X)=4$, and
from the certain full exceptional collections constructed in~\cite{Kuz96},~\cite{Or91},~\cite{Ka88},~\cite{Be78}
one can see that there exist a full exceptional collection that is an extension of the standard one,
so that
$$
\omega_\cO(X)=\omega(X)=\sum_{p=0}^{n} h^{p,p}(X)=4.
$$

Let us present some known estimates for $\omega(X)$ and $\omega_\cO(X)$ for some other Fano threefolds $X$ of Picard rank one.

\begin{lemma}
\label{lemma:omega for index 2}
Let $X$ be a smooth threefold from families $11$, $12$, $13$, $14$, that is $X$ is $Y_1$, $Y_2$, $Y_3$, or $Y_4$.
Then $\omega_\cO(X)=2$.
\end{lemma}

\begin{proof}
The proof is based on Kuznetsov's calculations in~\cite{Kuz09}.
By Proposition~\ref{proposition:standard collection}, one has $\Db(X)=\langle\cC,\cO_X,\cO_X(1)\rangle$
for some category $\cC$.
Let $K_0(X)$ be the Grothendieck group of coherent sheaves on $X$.
It is equipped with the Euler bilinear form
$$
\chi\colon K_0(X)\otimes K_0(X)\to \ZZ,\ \ \ \chi([F],[G])=\sum_i(-1)^i\dim \mathrm{Ext}^i(F,G).
$$
One can define a numerical Grothendieck group $K_0(X)_{\mathrm{num}}$ as $K_0(X)/\mathrm{Ker}\chi$.
The Chern character
$$\mathrm{ch}\colon K_0(X)\to
\HH=\bigoplus H^{2p}(X,\QQ)$$
induces an isomorphism between $K_0(X)_{\mathrm{num}}\otimes \QQ$ and
$\HH$.
Let $H$ be the generator of $\Pic(X)$ such that $-K_X=2H$, let $L$ be the class of a line on $X$, and let
$P$ be the class of a point on $X$.
Put $d=H^3$, so that $d=1,\ldots, 4$.
By~\cite[Corollary 5.8]{Kuz09} one has
$$
K_0(X)_{\mathrm{num}}=\langle[\cO_X],[\cO_H],[\cO_L],[\cO_P]\rangle.
$$
The Chern character identifies $K_0(X)_{\mathrm{num}}$ (equipped with the Euler form) with the lattice
$\LL\subset \HH
$ generated by
$$
\mathrm{ch}(\cO_X)=1,\ \ \mathrm{ch}(\cO_H)=H-\frac{d}{2}L+\frac{d}{6}P,\ \ \mathrm{ch}(\cO_L)=L,\ \ \mathrm{ch}(\cO_P)=P.
$$
This implies that
$$
K_0(\cC)_{\mathrm{num}}=\langle \mathrm{ch}([\cO_X]),\mathrm{ch}([\cO_X(1)])\rangle^\bot\subset \LL.
$$
By the proof of~\cite[Proposition 3.9]{Kuz09}, the Euler form on $K_0(\cC)_{\mathrm{num}}\simeq \ZZ^2$ is given by the matrix
$$
\left(
  \begin{array}{cc}
    -1 & -1 \\
    1-d & -d \\
  \end{array}
\right).
$$
It is easy to see that this form is non-positively definite. 
Since for an exceptional object $\mathcal{F}$
one has $\chi(\mathcal{F},\mathcal{F})=1$, there are no exceptional objects in $\cC$, which
gives the assertion of the lemma.
\end{proof}

\begin{remark}
For $d=4$, that is if $X$ is 
intersection of two quadrics, one can derive the assertion of Lemma~\ref{lemma:omega for index 2}
in another way. That is, by 
\cite[Theorem 2.9]{BO95},
one has
$$
\Db(X)\simeq\langle \Db(C),\cO_X,\cO_X(1)\rangle,
$$
where $C$ is a genus $2$ curve. However $\Db(C)$ does not have an exceptional object by \cite{Ok11}.
\end{remark}

\begin{proposition}[{\cite[Lemma 3.6]{Kuz09}, \cite[6.2]{Kuz06}, \cite[6.3]{Kuz06}}]
\label{proposition:index one two objects}
Let $X=X_n$ be a Fano threefold for $n=5,6,7,8,9$, that is $V_{10}$, $V_{12}$, $V_{14}$, $V_{16}$, or $V_{18}$. Then $\omega_\cO(X)\ge 2$.
\end{proposition}

Let us summarise the known results about $\omega_\cO(X)$, and, thus, estimates for $\omega(X)$, for smooth Picard rank
one Fano threefolds (providing references for them). For completeness we add their notation, anticanonical degree, the unique non-trivial Hodge number,
and a brief description for each family. For the family number $3$ we provide estimates for general and for
specific varieties in the family. 

	\small{
		\begin{longtable}{|c|c|c|c|p{7cm}|HHp{1.6cm}|c|}
			\caption{Estimates for 
$\omega_\cO(X)$ for Picard rank one Fano threefolds $X$.}\label{table:estimates}\\
			\hline Fam. & Not. & $-K^3$ &$h^{1,2}$&  Brief description & $\omega$ & Ref. & $\omega_\cO(X)$ & Reference \\
			\hline $1$ & $V_2$ & $2$ & $52$ & a hypersurface in $\mathbb{P}(1,1,1,1,3)$ of
			degree $6$ & $\ge 1 $ & Rem.~\ref{remark: omega ge omegaO} & $\ge 1$ & Rem.~\ref{remark: omega ge omegaO} \\
			\hline $2$ & $V_{4}$ & $4$&$30$ & a hypersurface in $\mathbb{P}^4$ of degree
			$4$ or\hfill\break
a double cover of smooth quadric in $\mathbb{P}^{4}$ branched over a surface of degree $8$ & $\ge 1 $ & Rem.~\ref{remark: omega ge omegaO} & $\ge 1$ & Rem.~\ref{remark: omega ge omegaO} \\
			\hline $3$ & $V_{6}$ & $6$ &$20$& a complete intersection of a quadric and a cubic in
			$\mathbb{P}^{5}$ & gen. $\ge 2$,\hfill\break sp. $\ge 1$ & Rem.~??? & gen. $\ge 2$,\hfill\break sp. $\ge 1$ & Prop.~\ref{proposition:quadric and cubic} \\
			\hline $4$ & $V_{8}$ & $8$ &$14$& a complete intersection of three quadrics in
			$\mathbb{P}^{6}$& $\ge 1 $ & Rem.~\ref{remark: omega ge omegaO} & $\ge 1$ & Rem.~\ref{remark: omega ge omegaO} \\
			\hline $5$ & $V_{10}$ & $10$ &$10$& a section of
			$\mathrm{Gr}(2,5)\subset\mathbb{P}^9$ by
			a quadric and a linear subspace of dimension~$7$ or \hfill\break
			a double cover of the variety 15 with the branch locus an anti-canonical divisor
			& $\ge 2 $ & Rem.~\ref{remark: omega ge omegaO} & $\ge 2$ & Prop.~\ref{proposition:index one two objects} \\
			\hline $6$ & $V_{12}$ & $12$ &$7$& a section of
half-spinor embedding of a connected component of orthogonal Lagrangian Grassmannian $\mathrm{OGr}_+(5,10)$ by codimension $5$ subspace
& $\ge 2$ & Rem.~\ref{remark: omega ge omegaO}& $\ge 2$ & Prop.~\ref{proposition:index one two objects} \\
			\hline $7$ & $V_{14}$ & $14$ &$5$ & a section of
			$\mathrm{Gr}(2,6)\subset\mathbb{P}^{14}$
			by a linear subspace of codimension~$5$			& $\ge 2 $ & Rem.~\ref{remark: omega ge omegaO} & $\ge 2$ & Prop.~\ref{proposition:index one two objects} \\

			\hline $8$ & $V_{16}$ & $16$ & $3$& a section of
symplectic Lagrangian Grassmannian $\mathrm{SGr}(3,6)\subset \PP^{13}$ by a linear subspace of codimension $3$
& $\ge 2$ & Rem.~\ref{remark: omega ge omegaO}& $\ge 2$ & Prop.~\ref{proposition:index one two objects} \\
			\hline $9$ & $V_{18}$ & $18$ & $2$& a section of
the adjoint $G_2$-Grassmannian $\mathrm{G_2Gr}(2,7)$ by codimension $2$ subspace
& $\ge 2 $ & Rem.~\ref{remark: omega ge omegaO} & $\ge 2$ & Prop.~\ref{proposition:index one two objects} \\

			\hline $10$ & $V_{22}$ & $22$ &$0$& a zero locus of three sections of the rank
			$3$ vector bundle $\bigwedge^2\mathcal{Q}$, 
where
			$\mathcal{Q}$ is
			the universal quotient bundle on $\mathrm{Gr}(3,7)$ & $4$ & \cite{Kuz96} & $4$ & \cite{Kuz96} \\
			\hline $11$ & $Y_{1}$ & $8$ &$21$& 
a hypersurface in
			$\mathbb{P}(1,1,1,2,3)$ of degree $6$& $\ge 2$ & Rem.~\ref{remark: omega ge omegaO} & $2$ &Lem.~\ref{lemma:omega for index 2}
			\\
			\hline $12$ & $Y_{2}$ & $16$ &$10$& 
a hypersurface in
			$\mathbb{P}(1,1,1,1,2)$ of degree $4$& $\ge 2$ & Rem.~\ref{remark: omega ge omegaO} & $2$ &Lem.~\ref{lemma:omega for index 2}
			\\
			\hline $13$ & $Y_{3}$ & $24$ &$5$& 
a hypersurface in
			$\mathbb{P}^{4}$ of degree $3$& $\ge 2$ & Rem.~\ref{remark: omega ge omegaO} & $2$ &Lem.~\ref{lemma:omega for index 2}
			\\
			\hline $14$ & $Y_{4}$ & $32$ &$2$& 
a complete intersection of two
			quadrics in $\mathbb{P}^{5}$ & $\ge 2$ & Rem.~\ref{remark: omega ge omegaO} & $2$ & Lem.~\ref{lemma:omega for index 2}
			\\
			\hline $15$ & $Y_{5}$ & $40$ &$0$& 
a section of
			$\mathrm{Gr}(2,5)\subset\mathbb{P}^9$ by a linear subspace of
			codimension $3$ & $4$ & \cite{Or91} & $4$ & ~\cite{Or91}\\
			\hline $16$ & $Q$ & $54$ &$0$& 
a hypersurface in $\mathbb{P}^{4}$
			of degree $2$ & $4$ & \cite{Ka88} & $4$ & \cite{Ka88} \\
			\hline $17$ & $\PP^3$ & $64$ &$0$& $\mathbb{P}^{3}$ & $4$ & \cite{Be78} & $4$ & \cite{Be78}\\
			\hline
		\end{longtable}
	}

\section{Landau--Ginzburg models}
\label{section:LG}

We briefly recall constructions and properties of
Landau--Ginzburg models; for details see~\cite{Prz18},~\cite{KP22}, and~\cite{DHKOP24}.

Let $X$ be a Picard rank one Fano threefold. One can associate a \emph{toric Landau--Ginzburg model} to $X$ --- a Laurent polynomial $f$ in $3$
variables, such that the following hold.
First, constant term of regularized generating series of genus $0$ Gromov--Witten invariants
$\widetilde{I}^X_0$ is equal to the main period $I_f$ of the family $\{f=\lambda\}$, $\lambda\in \CC$.
Second, Newton polytope of $f$ is equal to a fan polytope of some toric variety that is a degeneration
of $X$. Third, there should exist a \emph{(log) Calabi--Yau compactification} --- a fiberwise compactification
$(Y,w)$ of $((\CC^*)^3,f)$ such that $Y$ is a smooth Calabi--Yau variety and $w$ is proper (so that
its fibers are K3 surfaces by adjunction). If two toric Landau--Ginzburg models are \emph{mutationally equivalent},
then their Calabi--Yau compactifications differ by flops. This in particular means that the structure
of (a posteriori unique) reducible fiber of $w$, as well as a number of singular fibers, does not depend on a particular
Calabi--Yau compactification, see Lemma~\ref{lemma: odp for all} below. 
Moreover, if $-K_X$ is very ample, then one can associate to $X$ a unique mutation class of \emph{rigid maximally mutable}
Laurent polynomials containing a reflexive one; in the rest cases there exist a mutation class of rigid maximally mutable polynomials
called standard (while there is no reflexive polynomials in it). Alternatively, one can require that \emph{Dolgachev--Nikulin duality}
holds for $X$ and $(Y,w)$. In the Picard rank one case this means the following. Let $\Pic(X)=\ZZ H$, where $H$ is ample.
Then a general anticanonical section of $X$ is a Picard rank one K3 surface polarized by the lattice $\langle 2k\rangle$, where
$2k=H^3$. Then fibers of $w$ are K3 surfaces polarized by the lattice $M_k=U\oplus 2E_8(-1)\oplus \langle-2k\rangle$,
where $U$ is a hyperbolic lattice. Since there is a one-dimensional family of $M_k$-polarized K3 surfaces, one can imply that
$Y$ is $i_X$-to-one covering of this family, which gives uniqueness of $(Y,w)$.
We call any toric Landau--Ginzburg model whose Calabi--Yau compactification is $(Y,w)$ standard,
and $(Y,w)$ itself just a Landau--Ginzburg model for $X$ for short.

\begin{lemma}
\label{lemma: odp for all}
Let $(Y,w)$ be a Landau--Ginzburg model for a smooth 
Fano threefold $X$.
Let the fiber $w^{-1}(\lambda)$, $\lambda\in \CC$, have $k$ isolated singularities which are ordinary double points.
Then the fiber $(w')^{-1}(\lambda)$ of any other Landau--Ginzburg model $(Y',w')$ for $X$
also has $k$ isolated singular points which are ordinary double ones.
\end{lemma}

\begin{proof}
It is enough to prove the assertion of the lemma in the neighborhood of one singular point.
Let $p$ be an ordinary double point on $F=w^{-1}(\lambda)$.
Since anticanonical classes of both $(Y,w)$ and $(Y',w')$ are fibers, they are K-equivalent, so
standard Hironaka-type arguments shows that $Y$ and $Y'$ differ by flops.
More precise, by Weak factorization theorem, there is a diagram
$$
\xymatrix{
& V\ar[dl]_\pi \ar[dr]^{\pi'} & \\
Y\ar@{-->}[rr]^\varphi & & Y',
}
$$
where $\pi$ and $\pi'$ are given by a sequence of blow ups in smooth centers such that
exceptional divisors of $\pi$ and $\pi'$ coincide.

If $\varphi$ is an isomorphism in the neighborhood of $p$, then there is nothing to prove.
Thus, $\varphi$ is a flop in a curve $C$ passing through $p$.
Let $F'=(w')^{-1}(\lambda)$.
Let 
$C'\in F'$ be the flopping curve on $F'$, so that $\varphi$ is an isomorphism
between  $F\setminus C$ and
$F'\setminus C'$ in a neighborhood of $p$.
We may assume that $\pi=\pi_2\pi_1$, where $\pi_1
$ is a blow up of a curve 
passing through $p$.
Let $F_1$ be a strict transform of $F$. Then $F_1$ is a smooth K3 surface which is a minimal resolution
of $F'$, and in the neighborhood of $C$ the map $\pi'\pi_2^{-1}=\varphi\pi_1$ is a contraction
of $C$. Thus, $F'$ has a single ordinary double point in the neighborhood of $\varphi\pi_1(C)$.
\end{proof}

\begin{remark}
For the Picard rank one case one can derive the assertion of Lemma~\ref{lemma: odp for all} from Dolgachev--Nikulin
duality. That is, in the notation of the proof of the lemma, both $F$ and $F'$ are K3 surfaces of Picard rank $19$.
Note that $F'$ is singular, otherwise after the flop $\varphi^{-1}$ the surface $F$ is smooth.
However obviously a singular K3 surface of Picard rank $19$ can have at most a single ordinary double point.
\end{remark}

Thus, to compute $\varepsilon(X)$, one needs to pick up any appropriate toric Landau--Ginzburg model,
construct a Calabi--Yau compactification, and count the number of fibers with ordinary double points.
Universal construction of Calabi--Yau compactifications is given in~\cite{Prz17}. However the more convenient way
to keep track the singularities is presented in~\cite{ChP25}, see also~\cite{ChP20}.
We briefly recall it; all details are given in loc. cit.
If $X=X_n$ and $n\neq 1,11$, that is $V_2$ or $Y_1$, then one can choose a toric Landau--Ginzburg model
$f(x,y,z)$ for $X$ such that under the natural embedding
$$
(\CC^*)^3=\Spec \CC[x^{\pm 1}, y^{\pm 1}, z^{\pm 1}]\hookrightarrow \Proj \CC[x,y,z,t]=\PP^3
$$
the family $\{f=\lambda\}$, $\lambda\in \CC$, compactifies to a pencil $\cL_\lambda$ of quartics in $\PP^3$, where $\lambda\in \PP^1$ now. Thus the Calabi--Yau compactification
of $f$ is a (crepant over the base) resolution $\cN_\lambda$ of a base locus of $\cL_\lambda$.
(So that $\cN_\lambda=w^{-1}(\lambda)$.)
If $n=1,11$, then one can find a certain toric Landau--Ginzburg model that after a birational change of variables
again gives a pencil of quartics (see,~\cite[proof of Theorem 18]{Prz18}); we also denote it by $\cL_\lambda$.
The resolution of the base locus of $\cL_\lambda$ is given by a sequence of blow-ups of smooth curves. Thus,
the non-central fibers of $\cN_\lambda$ are birational to elements of $\cL_\lambda$.
Singularities of fibers of the resolved pencil lie over singularities of the elements of $\cL_\lambda$.
If the latter singular points do not lie on the base locus, then the resolution is
biregular in their neighborhood. Otherwise we need to resolve the singular locus
to realize if singularities of fibers appear or not. Thus to compute $\varepsilon(X)$
one needs to make a resolution of the base locus in the neighborhood of singularities of fibers of $\mathcal{L}_\lambda$.

\begin{lemma}
\label{lemma:P3}
Let $X=X_{17}$ be a projective space, and let $(Y,w)$ be a Landau--Ginzburg model for $X$.
Then $(Y,w)$ has four singular fibers, and each of them has a single ordinary double point.
Thus, $\varepsilon(X)=4$ and $(Y,w)$ is a Lefschetz pencil.
\end{lemma}

\begin{proof}
Let us take the toric Landau--Ginzburg model of Givental's type
$$
x+y+z+\frac{1}{xyz}.
$$
The family of its fibers can be compactified to the pencil of quartics given by
$$
(x+y+z)xyz+t^4=\lambda xyzt, \ \ \lambda \in \CC.
$$

Let $\alpha$ be $4$-th root of $1$. Singularities of members of the pencil are
\begin{multline*}
P_1^1=[1:0:0:0],\ \ P_1^2=[0:1:0:0],\ \ P_1^3=[0:0:1:0],\\
P_2^1=[1:-1:0:0],\ \
P_2^2=[1:0:-1:0], \ \
P_2^3=[0:1:-1:0],
\end{multline*}
as well as four ordinary double points $[\alpha:\alpha:\alpha:1]$ over $\lambda=4\alpha$ that do not lie at the base locus.
One can see that for any $\lambda$ the points $P_i^j$ are du Val of type $\dvA_3$.
Thus, we get $\varepsilon(X)=4$.
\end{proof}

\begin{lemma}[{cf.~\cite{Hu22}}]
\label{lemma:quadric}
Let $X=X_{16}$ be a quadric threefold, and let $(Y,w)$ be a Landau--Ginzburg model for $X$.
Then $(Y,w)$ has four singular fibers, and each of them has a single ordinary double point.
Thus, $\varepsilon(X)=4$ and $(Y,w)$ is a Lefschetz pencil.
\end{lemma}

\begin{proof}
Let us take the toric Landau--Ginzburg model of Givental's type
$$
\frac{(x+1)^2}{xyz}+y+z.
$$
The family of its fibers can be compactified to the pencil of quartics given by
$$
(x+t)^2t^2+(y+z)xyz=\lambda xyzt, \ \ \lambda \in \CC.
$$

Let $\alpha$ be a cubic root of $4$. Singularities of members of the pencil are
\begin{multline*}
P_1^1=[0:1:0:0],\ \ P_1^2=[0:0:1:0],\ \ P_2=[1:0:0:0],\ \ P_3=[1:0:0:-1],\\
P_4^1=[-1:\lambda:0:1], \ \
P_4^2=[-1:0:\lambda:1], \ \ P_5=[0:1:-1,0],
\end{multline*}
as well as thee ordinary double points $[1:\alpha:\alpha:1]$ over $\lambda=3\alpha$ that do not lie at the base locus.
One can see that for any $\lambda$ the points $P_1^1$ and $P_1^2$ are du Val of type $\dvA_3$,
the point $P_2$ is of type $\dvD_4$, the point $P_5$ is of type $\dvA_3$.
Thus they resolve after the resolution of the base locus.
The points $P_3$, $P_4^1$, and $P_4^2$ are du Val for $\lambda\neq 0$, and they collide when
$\lambda$ goes to $0$. Blowing $P_3$ up one gets $3$ singularities. First an ordinary double point over $0$
that does not lie on the base locus. Second, those that lie over $P_4^1$ and $P_4^2$ (in a generic fiber); they are du Val of type $\dvA_1$ for all $\lambda$, so they do not produce singularities of fibers.
Thus, we get $\varepsilon(X)=4$.
\end{proof}

Note that if $X$ is a threefold Fano complete intersection, but not a projective space or a quadric, then
$h^{1,2}(X)\neq 0$. 

\begin{theorem}[{see \cite[Theorem 22]{Prz13}
}]
Let $X$ be a Fano threefold, and let $(Y,w)$ be a Calabi--Yau compactification of 
toric Landau--Ginzburg
model for $X$. Then either $h^{1,2}(X)=0$, or the central fiber of $(Y,w)$ 
consists of $h^{1,2}(X)+1$ components.
\end{theorem}

Let $X$ now be an $n$-dimensional Fano complete intersection. Then one can construct so-called \emph{Givental's toric Landau--Ginzburg model},
see~\cite{PSh15} for details. (In the threefold case they are standard.) Its Calabi--Yau compactification has a fiber called central. It usually has non-isolated
singularities; say, it is reducible if and only if $h^{1,n-1}(X)\neq 0$, see~\cite[Theorem 1.2]{PSh15}.

\begin{theorem}[{\cite[Theorem 1.2]{Prz25}}]
Let $X$ be a Fano complete intersection. 
Then except for the central fiber of a Calabi--Yau compactification of Givental's toric Landau--Ginzburg model of $X$, there are
exactly $i_X$ singular fibers, and they have a single ordinary double point in each.
\end{theorem}

\begin{corollary}
\label{corollary: fibers for WCI}
Let $X$ be a smooth Picard rank one Fano threefold. If $n=2,3,4$, then $\varepsilon(X)=1$;
if $n=13,14$, then $\varepsilon(X)=2$. 
\end{corollary}

To prove Theorem~\ref{theorem:ODP for LG}, one needs to compute $\varepsilon(X)$ case-by-case by hand.
Let us sketch the computations in the following series of lemmas. We use our usual notation in them.

\begin{lemma}
\label{lemma: V22}
Let $X=X_{10}$, that is $V_{22}$, and let $(Y,w)$ be a Landau--Ginzburg model for $X$. Then $(Y,w)$ has $4$ singular fibers, and all of them have a single ordinary double point.
Thus, $\varepsilon(X)=4$ and $(Y,w)$ is a Lefschetz pencil.
\end{lemma}

\begin{proof}
Let us take the toric Landau--Ginzburg model
$$
\frac{y}{z}+\frac{x}{z}+x+y+z+\frac{1}{x}+\frac{1}{y}+\frac{1}{z}+\frac{xy}{z}+\frac{1}{xy}+\frac{z}{x}+\frac{z}{y}
$$
from~\cite{Prz13}.
The family of its fibers can be compactified to the pencil of quartics given by
$$
xy^2t+x^2yt+x^2yz+zy^2z+xyz^2+yzt^2+xzt^2+xyt^2+x^2y^2+zt^3+yz^2t+xz^2t=\lambda xyzt, \ \ \lambda \in \CC.
$$

Consider $\alpha$ such that $\alpha^3+2\alpha^2=2$. Singularities of members of the pencil are
\begin{multline*}
P_1^1=[1:0:0:0],\ \ P_1^2=[0:1:0:0],\ \ P_2^1=[1:0:0:-1], \\ P_2^2=[0:1:0:-1], \ \ P_3=[0:0:1:0],
\end{multline*}
as well as three ordinary double points $[1:1:1:\alpha]$ over $\lambda=3\alpha^2+8\alpha+8$,
the point $P_4=[1:1:-1:0]$ over $\lambda=0$ of type $\dvA_2$ that lies in the base locus, the point $P_5=[1:1:0:-1]$ over $\lambda=1$
of type $\dvA_2$ that lies in the base locus,
and ordinary double points $P_5^1=[1:0:1:-1]$, $P_5^2=[0:1:1:-1]$ over $\lambda=2$ that lie in the base locus.

Let us consider the singular points one-by-one.
The points $P_1^1$ and $P_1^2$ are of type $\dvA_3$ if $\lambda\neq 1$, and of type $\dvA_4$ otherwise.
Thus, each of them resolves by two blow-ups for any $\lambda$.
Similarly, the points $P_2^1$ and $P_2^2$ are of type $\dvA_1$ if $\lambda\neq 2$, and of type $\dvA_2$ otherwise,
so each of them resolves by one blow-up for any $\lambda$.
The point $P_3$ after blowing the base curve $\{x=y+t=0\}$ up gives a singular point that is of type $\dvA_1$
for $\lambda \neq 1$ and of type $\dvA_2$ for $\lambda=1$; thus, it resolves after one more blowing up.
The point $P_4$ gives an ordinary double point after blowing up the line $\{t=y+z=0\}$ that is a component of the base locus;
this ordinary double point does not lie on the base locus after blowing up. Since $P_5^1$ and $P_5^2$ are ordinary double
points, they resolve after blowing up any component of the base locus passing through them.
Summarizing, we get the assertion of the lemma. 
\end{proof}

\begin{lemma}
\label{lemma:V5}
Let $X=X_{15}$, that is $Y_5$, and let $(Y,w)$ be a Landau--Ginzburg model for $X$.
Then $(Y,w)$ has four singular fibers, and each of them has a single ordinary double point.
Thus, $\varepsilon(X)=4$ and $(Y,w)$ is a Lefschetz pencil.
\end{lemma}

\begin{proof}
Let us take the toric Landau--Ginzburg model
$$
x+y+z+\frac{1}{x}+\frac{1}{y}+\frac{1}{z}++\frac{1}{xyz}
$$
from~\cite{Prz13}.
The family of its fibers can be compactified to the pencil of quartics given by
$$
x^2yz+xy^2z+xyz^2+yzt^2+xzt^2+xyt^2+t^4=\lambda xyzt, \ \ \lambda \in \CC.
$$

Let $\alpha$ be such that $\alpha^4+\alpha^2=1$. Singularities of members of the pencil are
\begin{multline*}
P_1^1=[1:0:0:0],\ \ P_1^2=[0:1:0:0],\ \ P_1^3=[0:0:1:0],\\
P_2^1=[1:-1:0:0], \ \ P_2^2=[1:0:-1:0], \ \ P_2^3=[0:1:-1,0],
\end{multline*}
as well as four ordinary double points $[1:1:1:\alpha]$ over $\lambda=4\alpha^2+6\alpha$ that do not lie at the base locus.
One can see that for any $\lambda$ the points $P_1^i$ are du Val of type $\dvA_3$,
and $P_2^i$ are of type $\dvA_1$.
Blowing them up we get $\varepsilon(X)=4$.
\end{proof}

The rest varieties have non-trivial intermediate Jacobian, so to find $\varepsilon(X)$,
one needs to study non-central fibers only.

\begin{lemma}
\label{lemma:V2}
Let $X=X_1$, that is $V_2$, and let $(Y,w)$ be a Landau--Ginzburg model for $X$.
Then $(Y,w)$ has two singular fibers: one reducible fiber and one fiber having a single ordinary double point.
Thus, $\varepsilon(X)=1$.
\end{lemma}

\begin{proof}
A family given by a standard toric Landau--Ginzburg model is birationally
(over the base) equivalent to the pencil of quartics
$$
x^4=\lambda yz(xt-xy-xz-t^2), \ \ \lambda \in \CC,
$$
see~\cite[proof of Theorem 18]{Prz13}.
Singularities of members of the pencil are
\begin{multline*}
P_1=[0:0:0:1],\ \ P_2^1=[0:1:0:0],\ \ P_2^2=[0:0:1:0], \ \
P_3=[0:1:-1:0],
\end{multline*}
the cental fiber (over $\lambda=0$), and the ordinary double point $[12:1:1:6]$ over $\lambda=1728$ that do not lie at the base locus.
One can see that for any $\lambda\neq 0$ the point $P_1$ is du Val of type $\dvA_3$,
the points $P_2^i$ are of type $\dvA_6$, and the point $P_3$ is of type $\dvA_1$.
This gives the assertion of the lemma.
\end{proof}

\begin{lemma}
\label{lemma:V10}
Let $X=X_5$, that is $V_{10}$, and let $(Y,w)$ be a Landau--Ginzburg model for $X$.
Then $(Y,w)$ has three singular fibers: one reducible fiber and two fibers having a single ordinary double point.
\end{lemma}

\begin{proof}
Let us take the toric Landau--Ginzburg model
$$
\frac{(1+x+y+z+xy+xz+yz)^2}{xyz}
$$
from~\cite{Prz13}.
The family of its fibers can be compactified to the pencil of quartics given by
$$
(t^4+xt^3+yt^3+zt^3+xyt^2+xzt^2+yzt^2)^2=\lambda xyzt, \ \ \lambda \in \CC.
$$
Let $\alpha^2+\alpha=1$. Singularities of members of the pencil are
\begin{multline*}
P_1^1=[1:-1:0:0],\ \ P_1^2=[1:0:-1:0],\ \ P_1^3=[1:0:0:-1],\ \ P_1^4=[0:1:-1:0], \\
P_2^1=[1:0:0:0],\ \ P_2^2=[0:1:0:0],\ \ P_2^3=[0:0:1:0],\ \ P_2^4=[0:0:0:1],
\end{multline*}
the cental fiber (over $\lambda=0$), and two ordinary double points $[1:1:1:\alpha]$ over $\lambda=20\alpha+32$ that do not lie at the base locus.
One can see that for any $\lambda\neq 0$ the points $P_1^i$ are du Val of type $\dvA_1$,
and the points $P_2^i$ are of type $\dvA_2$.
This gives the assertion of the lemma.
\end{proof}

\begin{lemma}
\label{lemma:V12}
Let $X=X_6$, that is $V_{12}$, and let $(Y,w)$ be a Landau--Ginzburg model for $X$.
Then $(Y,w)$ has three singular fibers: one reducible fiber and two fibers having a single ordinary double point.
Thus, $\varepsilon(X)=2$.
\end{lemma}

\begin{proof}
Let us take the toric Landau--Ginzburg model
$$
\frac{(x+z+1)(x+y+z+1)(z+1)(y+z)}{xyz}
$$
from~\cite{Prz13}.
The family of its fibers can be compactified to the pencil of quartics given by
$$
(x+z+t)(x+y+z+t)(z+t)(y+z)=\lambda xyzt, \ \ \lambda \in \CC.
$$
Consider $\alpha$ such that $\alpha^2+2\alpha=1$. Singularities of members of the pencil are
\begin{multline*}
P_1=[1:0:0:0],\ \ P_2=[0:1:0:0],\ \ P_3=[1:-1:0:0],\ \ P_4=[1:0:0:-1], \\
P_5=[1:0:-1:0],\ \ P_6=[0:1:-1:0],\ \ P_7=[0:0:1:-1],
\end{multline*}
the cental fiber (over $\lambda=0$), and two ordinary double points $[2\alpha:\alpha+1:\alpha-1:2]$ over
$\lambda=12\alpha+5    
$ that do not lie at the base locus.
One can see that for any $\lambda\neq 0$ the points $P_1,P_2,P_4,P_5,P_7$ are du Val of types $\dvA_2,\dvA_1,\dvA_3,\dvA_1,\dvA_2$,
correspondingly.
The points $P_3,P_6$ are du Val of type $\dvA_1$ for $\lambda\neq 0,1$ and of type $\dvA_2$ for $\lambda=1$;
thus, they resolve by one blowing up for all $\lambda\neq 0$.
This gives the assertion of the lemma.
\end{proof}

\begin{lemma}
\label{lemma:V14}
Let $X=X_7$, that is $V_{14}$, and let $(Y,w)$ be a Landau--Ginzburg model for $X$.
Then $(Y,w)$ has three singular fibers: one reducible fiber and two fibers having a single ordinary double point.
Thus, $\varepsilon(X)=2$.
\end{lemma}

\begin{proof}
Let us take the toric Landau--Ginzburg model
$$
f=\frac{(x+y+z+1)^2}{x} + \frac{(x + y + z + 1)(y + z + 1)(z + 1)^2}{xyz}
$$
from~\cite{Prz13}.
The family of its fibers can be compactified to the pencil of quartics given by
$$
(x+y+z+t)^2yz + (x+y+z+t)(y+z+t)(z+t)^2=\lambda xyzt, \ \ \lambda \in \CC.
$$
The singularities of members of the pencil are
\begin{multline*}
P_1=[0:1:0:-1],\ \ P_2=[0:0:1:-1],\ \ P_3=[\lambda:0:-1:1], \\ P_4=[1:0:0:0], \ \
P_5=[1:-1:0:0],\ \ P_6=[0:1:-1:0],
\end{multline*}
the cental fiber (over $\lambda=0$), and the ordinary double point $[12:3:1:2]$ over
$\lambda=27$ that do not lie at the base locus.
One can see that for any $\lambda\neq 0$ the points $P_1, P_2, P_3,P_4,P_5$ are du Val of types $\dvA_1,\dvA_3,\dvA_2,\dvA_2,\dvA_3$,
correspondingly.
The point $P_6$ lies on the base curve given by $x=y+z+t=0$. After blowing this curve up
the point has type $\dvA_1$ for $\lambda\neq 0,-1$ and of type $\dvA_3$ for $\lambda=-1$.
Blowing it up (once) one gets a singular fiber over $\lambda=-1$ with a single ordinary double point.
This gives the assertion of the lemma.
\end{proof}

\begin{lemma}
\label{lemma:V16}
Let $X=X_8$, that is $V_{16}$, and let $(Y,w)$ be a Landau--Ginzburg model for $X$.
Then $(Y,w)$ has three singular fibers: one reducible fiber and two fibers having a single ordinary double point.
Thus, $\varepsilon(X)=2$.
\end{lemma}

\begin{proof}
Let us take the toric Landau--Ginzburg model
$$
\frac{(x+y+z+1)(x+1)(y+1)(z+1)}{xyz}
$$
from~\cite{Prz13}.
The family of its fibers can be compactified to the pencil of quartics given by
$$
(x+y+z+t)(x+t)(y+t)(z+t)=\lambda xyzt, \ \ \lambda \in \CC.
$$
Consider $\alpha$ such that $\alpha^2+2\alpha=1$. Singularities of members of the pencil are
\begin{multline*}
P_1^1=[1:0:0:0],\ \ P_1^2=[0:1:0:0],\ \ P_1^3=[0:0:1:0],\\
P_2^1=[1:-1:0:0], \ \ P_2^2=[1:0:-1:0],\ \ P_2^3=[0:1:-1:0],\\
P_3^1=[1:0:0:-1], \ \ P_3^2=[0:1:0:-1],\ \ P_3^3=[0:0:1:-1],
\end{multline*}
the cental fiber (over $\lambda=0$), and two ordinary double points $[1:1:1:\alpha]$ over
$\lambda=8\alpha+20     
$ that do not lie at the base locus.
One can see that for any $\lambda\neq 0$ the points $P_1^i$ are du Val of type $\dvA_2$,
and  the points $P_2^i$ are du Val of type $\dvA_1$.
The points $P_3^i$ are of type $\dvA_1$ for $\lambda\neq 0,1$ and of type $\dvA_2$ for $\lambda=1$;
thus, they resolve by one blowing up for all $\lambda\neq 0$.
This gives the assertion of the lemma.
\end{proof}

\begin{lemma}
\label{lemma:V18}
Let $X=X_9$, that is $V_{18}$, and let $(Y,w)$ be a Landau--Ginzburg model for $X$.
Then $(Y,w)$ has three singular fibers: one reducible fiber and two fibers having a single ordinary double point.
Thus, $\varepsilon(X)=2$.
\end{lemma}

\begin{proof}
Let us take the toric Landau--Ginzburg model
$$
\frac{(x+y+z)(x+y+z+xy+xz+yz+xyz)}{xyz}
$$
from~\cite{Prz13}.
The family of its fibers can be compactified to the pencil of quartics given by
$$
(x+y+z)(xt^2+yt^2+zt^2+xyt+xzt+yzt+xyz)=\lambda xyzt, \ \ \lambda \in \CC.
$$
Let $\alpha=\frac{\pm 1}{\sqrt{3}}$. Then
the singularities of members of the pencil are
\begin{multline*}
P_1^1=[1:0:0:0],\ \ P_1^2=[0:1:0:0],\ \ P_1^3=[0:0:1:0],\\
P_2^1=[1:-1:0:0], \ \ P_2^2=[1:0:-1:0],\ \ P_2^3=[0:1:-1:0], \ \ P_3=[0:0:0:1],
\end{multline*}
the cental fiber (over $\lambda=0$), and two ordinary double points $[1:1:1:\alpha]$ over
$\lambda=18\alpha+9  
$ that do not lie at the base locus.
One can see that for any $\lambda\neq 0$ the points $P_2^i$ are du Val of type $\dvA_1$,
and  the point $P_3$ is du Val of type $\dvD_4$.
The points $P_1^i$ are of type $\dvA_2$ for $\lambda\neq 0,-1$ and of type $\dvA_3$ for $\lambda=-1$;
they also lie on intersections of two components of base locus.
One can see that after a resolution of the base locus there is no singularity over $P_1^i$
for all $\lambda\neq 0$. Another way to see this is to consider the pencil of quartics
as a (singular) hypersurface in $\PP(x:y:z:t)\times \Aff(\lambda)$; after blowing $P_1^i\times \Aff(\lambda)$ up
one gets no singularity over $P_1^i$ for $\lambda\neq 0$, except for an ordinary double point on the hypersurface
given by $\lambda=-1$. This ordinary double point admit a projective small resolution, which also gives the assertion of the lemma.
\end{proof}

\begin{lemma}
\label{lemma:X1}
Let $X=X_{11}$, that is $Y_1$, and let $(Y,w)$ be a Landau--Ginzburg model for $X$.
Then $(Y,w)$ has three singular fibers: one reducible fiber and two fibers having a single ordinary double point.
Thus, $\varepsilon(X)=2$.
\end{lemma}

\begin{proof}
A family given by a standard toric Landau--Ginzburg model is birationally
(over the base) equivalent to the pencil of quartics
$$
x^4+z^2(xy-xt-t^2)=\lambda yz(xy-xt-t^2), \ \ \lambda \in \CC,
$$
see~\cite[proof of Theorem 18]{Prz13}.
Let $\alpha=\pm \sqrt{3}$.
Singularities of members of the pencil are
\begin{multline*}
P_1=[0:0:0:1],\ \ P_2=[0:1:0:0],\ \ P_3=[0:0:1:0], \ \
P_4=[0:1:\lambda:0],
\end{multline*}
the cental fiber (over $\lambda=0$), and two ordinary double points $[6:1:12\alpha:3]$ over $\lambda=24\alpha$ that do not lie at the base locus.
One can see that for any $\lambda\neq 0$ the point $P_1$ is du Val of type $\dvA_3$,
the points $P_2$ is of type $\dvA_6$, the point $P_3$ is of type $\dvA_1$,
and the point $P_4$ is of type $\dvA_6$.
This gives the assertion of the lemma.
\end{proof}

\begin{lemma}
\label{lemma:X2}
Let $X=X_{12}$, that is $Y_2$, and let $(Y,w)$ be a Landau--Ginzburg model for $X$.
Then $(Y,w)$ has three singular fibers: one reducible fiber and two fibers having a single ordinary double point.
Thus, $\varepsilon(X)=2$.
\end{lemma}

\begin{proof}
Let us take the toric Landau--Ginzburg model
$$
\frac{(x+y+1)^4}{xyz}+z
$$
from~\cite{Prz13}.
The family of its fibers can be compactified to the pencil of quartics given by
$$
(x+z+t)^4+xyz^2=\lambda xyzt, \ \ \lambda \in \CC.
$$
Singularities of members of the pencil are
\begin{multline*}
P_1^1=[1:0:0:-1],\ \ P_1^2=[0:1:0:-1],\ \ P_2=[0:0:1:0],\\
P_3^1=[-1:0:\lambda:1], \ \
P_3^2=[0:-1:\lambda:1],\ \ P_4=[1:-1:0:0],
\end{multline*}
the cental fiber (over $\lambda=0$), and two ordinary double points $[1:1:\lambda:2]$ over
$\lambda=\pm 16$ that do not lie at the base locus.
One can see that for any $\lambda\neq 0$ the points $P_1^i,P_2,P_3^i,P_4$ are du Val of type~$\dvA_3$.
This gives the assertion of the lemma.
\end{proof}

Summing the statements in this section up, one gets the assertion of Theorem~\ref{theorem:ODP for LG}.

\begin{proof}[Proof of Theorem~\ref{theorem:ODP for LG}]
See Lemmas~\ref{lemma:P3},~\ref{lemma:quadric},~\ref{lemma: V22},~\ref{lemma:V5},~\ref{lemma:V2},~\ref{lemma:V10},~\ref{lemma:V12},~\ref{lemma:V14},~\ref{lemma:V16},~\ref{lemma:V18},~\ref{lemma:X1},~\ref{lemma:X2},
and Corollary~\ref{corollary: fibers for WCI}.
\end{proof}


\begin{thebibliography}{999}

\bibitem[AAK16]{AAK16}
M.\,Abouzaid, D.\,Auroux, L.\,Katzarkov,
\emph{Lagrangian fibrations on blowups of toric varieties and mirror symmetry for hypersurfaces},
Publ. Math. Inst. Hautes \'Etudes Sci., \textbf{123} (2016), 199--282.

\bibitem[AKO06]{AKO06}
D.\,Auroux, L.\,Katzarkov, D.\,Orlov,
{\it Mirror symmetry for del Pezzo surfaces: vanishing cycles and coherent sheaves},
Invent. Math., \textbf{166}:3 (2006), 537-–582.

\bibitem[Ba02]{Ba02}
P.\,Balmer, {\it Presheaves of triangulated categories and reconstruction of schemes}, Math. Ann. \textbf{324} (2002), 557--580.

\bibitem[Be78]{Be78}
A. \,A.\,Beilinson, {\it Coherent sheaves on Pn and problems of linear algebra}, Funct. Anal. Appl., \textbf{12}:3 (1978), 214--216.

\bibitem[BO95]{BO95}
A.\,Bondal, D.\,Orlov,
{\it Semiorthogonal decompositions for algebraic varieties}, arXiv: alg-geom/9506012.

\bibitem[BO01]{BO01}
A.\,Bondal, D.\,Orlov,
{\it Reconstruction of a variety from the derived category and groups of autoequivalences},
Compos. Math., \textbf{125}:3 (2001), 327--344.

\bibitem[ChP22]{ChP20}
I.\,Cheltsov, V.\,Przyjalkowski, {\it Fibers over infinity of Landau--Ginzburg models}, Commun. Number Theory Phys., \textbf{16}:4 (2022), 673--693.

\bibitem[CP25]{ChP25}
I.\,Cheltsov, V.\,Przyjalkowski, {\it Katzarkov--Kontsevich--Pantev conjecture for Fano threefolds},
Proc. Steklov Inst. Math., \textbf{328} (2025), 1--156.

\bibitem[DHKOP24]{DHKOP24}
C.\, Doran, A.\, Harder, L.\, Katzarkov, M.\, Ovcharenko,
V.\, Przyjalkowski,
\emph{Modularity of Landau--Ginzburg models}, Proc. Steklov Inst. Math., \textbf{328} (2025), 157--295.


\bibitem[GPS20]{GPS20}
S.\,Ganatra, J.\,Pardon, V.\,Shende, \emph{Covariantly functorial wrapped Floer
theory on Liouville sectors}, Publ. Math. Inst. Hautes \'Etudes Sci., \textbf{131} (2020), 73--200.

\bibitem[GPS24a]{GPS24a}
S.\,Ganatra, J.\,Pardon, V.\,Shende, \emph{Microlocal Morse theory of wrapped Fukaya categories}, Ann. of Math. (2), \textbf{199}:3
(2024), 943--1042.

\bibitem[GPS24b]{GPS24b}
S.\,Ganatra, J.\,Pardon, V.\,Shende, \emph{Sectorial descent for wrapped Fukaya categories}, J. Amer. Math. Soc., \textbf{37}:2
(2024), 499--635.

\bibitem[Hu22]{Hu22}
X.\,Hu, \emph{Mirror symmetry for quadric hypersurfaces}, arXiv:2204.07858.

\bibitem[ILP13]{ILP13}
N.\,Ilten, J.\,Lewis, V.\,Przyjalkowski, \emph{Toric degenerations of Fano threefolds giving weak Landau--Ginzburg models},
J. Algebra, \textbf{374} (2013), 104--121.

\bibitem[Is78]{Is78}
V.\,A.\,Iskovskih, {\it Fano threefolds. II}, Izv. Akad. Nauk SSSR Ser. Mat., \textbf{42}:3 (1978), 506--549.

\bibitem[IP99]{IP99}
V.\,Iskovskikh, Yu.\,Prokhorov, {\it Fano varieties},
Encyclopaedia Math. Sci., \textbf{47}, Springer, Berlin, 1999.

\bibitem[Ka88]{Ka88}
M.\,Kapranov, {\it On the derived categories of coherent sheaves on some homogeneous spaces}, Invent.
Math., \textbf{92}:3 (1988), 479--508.

\bibitem[KP22]{KP22}
A.\,Kasprzyk, V.\,Przyjalkowski,
\emph{Laurent polynomials in mirror symmetry: why and how?},
Proyecciones, \textbf{41}:2 (2022), 481--515.

\bibitem[KO95]{KO95}
S.\,Kuleshov, D.\,Orlov, {\it Exceptional sheaves on del Pezzo surfaces}, Russian Acad. Sci. Izv. Math., \textbf{44}:3 (1995), 479--513.

\bibitem[Kuz96]{Kuz96}
A.\,Kuznetsov, {\it An exceptional set of vector bundles on the varieties $V_{22}$}, Vestnik Moskov. Univ.
Ser. I Mat. Mekh. \textbf{92}:3 (1996), 41--44.

\bibitem[Kuz06]{Kuz06}
A.\,Kuznetsov, {\it Hyperplane sections and derived categories},
Izv. Math. \textbf{70}:3 (2006), 447--547.

\bibitem[Kuz09a]{Kuz09}
A.\,Kuznetsov, {\it Derived categories of Fano threefolds},
Proc. Steklov Inst. Math, \textbf{264} (2009), 110--122.

\bibitem[Kuz09b]{Kuz09b}
A.\,Kuznetsov, {\it Hochschild homology and semiorthogonal decompositions}, arXiv:0904.4330.

\bibitem[Kuz13]{Kuz13}
A.\,Kuznetsov, {\it A simple counterexample to the Jordan-H\"older property for derived categories}, arXiv:1304.0903.

\bibitem[KP21]{KP21}
A.\,Kuznetsov, A.\,Perry, {\it Serre functors and dimensions of residual categories}, ArXiv:math.AG/2109.02026.

\bibitem[KS25]{KS25}
A.\,Kuznetsov, E.\,Shinder,
{\it Derived categories of Fano threefolds and degenerations},
Invent. Math. \textbf{239}:2 (2025), 377--430.

\bibitem[LP25]{LP25}
S.\,Lee, V.\,Przyjalkowski,
{\it Fibers of Landau--Ginzburg models and rationality}, arXiv:2510.21222.


\bibitem[Ok11]{Ok11}
Sh.\,Okawa,
{\it Semi-orthogonal decomposability of the derived category of a curve},
Adv. Math. \textbf{228}:5 (2011), 2869--2873.

\bibitem[Or91]{Or91}
D.\,Orlov, {\it Exceptional set of vector bundles on the variety $V_5$}, Vestnik Moskov. Univ. Ser. I Mat.
Mekh., \textbf{5} (1991), 69--71.

\bibitem[Prz13]{Prz13}
V.\,Przyjalkowski, \emph{Weak Landau--Ginzburg models for smooth Fano threefolds}, Izv. Math. \textbf{77}:4 (2013), 135--160.

\bibitem[Prz17]{Prz17}
V.\,Przyjalkowski, \emph{Calabi--Yau compactifications of toric Landau--Ginzburg models for smooth Fano threefolds}, Sb. Math.,
\textbf{208}:7 (2017), 992--1013.

\bibitem[Prz18]{Prz18}
V.\,Przyjalkowski, \emph{Toric Landau--Ginzburg models}, Russ. Math. Surv., \textbf{73}:6 (2018), 1033--1118.

\bibitem[Prz25]{Prz25}
V.\,Przyjalkowski, {\it Singularities of Landau-Ginzburg models for complete intersections and derived categories}, arXiv:2510.23143.

\bibitem[PSh15]{PSh15}
V.\,Przyjalkowski, C.\,Shramov, {\it On Hodge numbers of complete intersections and Landau--Ginzburg models},
Int. Math. Res. Not. IMRN, \textbf{2015}:21 (2015), 11302--11332.

\bibitem[Se08]{Se08}
P.\,Seidel,
\emph{Fukaya categories and Picard-Lefschetz theory},
Zur. Lect. Adv. Math. European Mathematical Society (EMS), Z\"{u}rich, 2008.



\end{thebibliography}
\end{document}